\documentclass[12pt,a4paper,%headinclude,
fleqn,reqno]{amsart}                 
\usepackage[T1]{fontenc}                   
\usepackage[utf8]{inputenc}                 
\usepackage[english]{babel}    
 \usepackage{csquotes}
\usepackage{graphicx}                      % 
\usepackage[font=small]{quoting}    
\usepackage{float}
\usepackage{caption}  
\usepackage[top=1.2in,bottom=1.8in,left=1.in,right=1.in]{geometry}                  
\usepackage{verbatim}
\usepackage{color}
\usepackage{xcolor}
\usepackage{picture}
\usepackage{amsmath,amsthm,amsfonts,amssymb}
\usepackage{comment}

\usepackage{hyperref}
\hypersetup{colorlinks,linkcolor={blue}, citecolor={blue},urlcolor={red}}  
\usepackage{enumitem}
\usepackage[url=false,doi=false,isbn=false, giveninits=true,style=numeric,maxbibnames=99]{biblatex}
\newtheorem{thm}{Theorem}[section] 
\newtheorem{lem}[thm]{Lemma} 
\newtheorem{cor}[thm]{Corollary} 
\newtheorem{prop}[thm]{Proposition}

\theoremstyle{definition} 
\newtheorem{rem}[thm]{Remark} 

\theoremstyle{remark}

\numberwithin{equation}{section}

\def\S{\Sigma} 
\def\n{\nabla}

\def\p{\partial}

\def\n{\nabla}

\def\p{\partial}

\def\k{\kappa}

\def\s{\sigma}

\def\ov{\overline}

\def\n{\nabla}
\def\<{\langle}
\def\>{\rangle}
\def\div{{\rm div}}

\def\n{\nabla}

\def\RR{\mathbb{R}}

\def\SS{\mathbb{S}}

\def\p{\partial}

\def\s{\sigma}

\def\ov{\overline}

\def\wh{\widehat}
\def\S{\Sigma}

\def\V{\mathcal V}

\def\R{\mathbb{R}}
\def\I{\mathcal{I}}
\def\C{\mathcal{C}}
\def\T{\mathcal{T}}
\def\L{\mathcal{L}}

\def\ol{\overline}
\def\wh{\widehat}

\def\bt{\mathbf{b}_\theta}

{\left\{\begin{array}{@{}l@{}}}{\end{array}\right.}
\patchcmd{\abstract}{\scshape\abstractname}{\textbf{\abstractname}}{}{}
\makeatletter 
\def\@makefnmark{} 
\makeatother

\numberwithin{equation}{section}
\numberwithin{exa}{section}

\begin{document}	
	
\title[A fully nonlinear flow]{A fully nonlinear locally constrained curvature  flow for capillary hypersurface}

\author[X. Mei]{Xinqun Mei}
\address[X. Mei]{Mathematisches Institut, Albert-Ludwigs-Universit\"{a}t Freiburg, Freiburg im Breisgau, 79104, Germany}

\email{\href{mailto:xinqun.mei@math.uni-freiburg.de} {xinqun.mei@math.uni-freiburg.de}}

\author[L. Weng]{Liangjun Weng}
\address[L. Weng]{Dipartimento di Matematica, Universita di Pisa, Pisa, 56127, Italy. Dipartimento Di Matematica, Università degli Studi di Roma "Tor Vergata", Roma, 00133, Italy}

 \email{\href{mailto:liangjun.weng@uniroma2.it}  {liangjun.weng@uniroma2.it}}

	% Please provide a minimum of 5 keywords.

\begin{abstract}
In this article, we study a locally constrained fully nonlinear curvature flow for convex capillary hypersurfaces in half-space. We prove that the flow preserves the convexity,  exists for all time, and converges smoothly to a spherical cap. This can be viewed as the fully nonlinear counterpart of the result in  \cite{MWW}. 	As a byproduct, a high-order capillary isoperimetric ratio  \eqref{iso ratio}  evolves monotonically along this flow, which yields a class of the Alexandrov-Fenchel inequalities.
\end{abstract}
\maketitle	

 \section{Introduction}
In this paper,  we are concerned with constructing the geometric curvature flows to solve isoperimetric type problems for \textit{capillary hypersurfaces} in the half-space. A smooth, compact, embedded hypersurface in $\ov{{\R}_+^{n+1}}$  with boundary supported on $\p{{\R}_+^{n+1}}$ is called capillary hypersurface if it intersects with $\p{{{\R}}_+^{n+1}}$ at a constant \textit{contact angle} $\theta\in (0,\pi)$.  The primary objective of this article is that design a more general curvature flow to study the  Alexandrov-Fenchel inequalities for capillary hypersurfaces in \cite{WWX1}, i.e., flow \eqref{flow with capillary} below, which is a fully nonlinear type flow compared to our previous work together with Wang in \cite{MWW}  on a mean curvature type flow, a semi-linear flow.

Such kind of fully nonlinear type flow is inspired  by the seminal works of Guan-Li  \cite{GL09, GL15, GL18, GL21} for 
	 closed hypersurfaces and is constructed  by using  the Minkowski formulas \cite[Proposition 2.5]{WWX1}  for capillary hypersurface $\S$ in $\ov{\RR_{+}^{n+1}}$, i.e., 
	\begin{eqnarray}\label{minkowski formula0}
		\int_{\Sigma}H_{k-1}\left(1+\cos\theta\<\nu, e\>\right)dA= \int_{\Sigma}H_{k}\<x, \nu\>dA,
	\end{eqnarray}
	where $\nu$ is the unit normal vector of $\S\subset \ov{\R^{n+1}_+}$ and $x$ is the position vector of $\S$ in $\ov{\R^{n+1}_+}$.

The simplest example of capillary hypersurfaces in half-space is the family of the spherical caps lying entirely in $\ov{\RR^{n+1}_{+}}$ and intersecting $\partial {{\RR}^{n+1}_{+}}$ with a constant {contact angle} $\theta\in (0,\pi)$, and it is given by
	\begin{eqnarray}\label{sph-cap}
		\C_{\theta, r}(e):= \left\{x\in \ov{\RR^{n+1}_{+}} \Big |~ |x-r\cos\theta e|=r \right\},~r~\in [0, \infty),
	\end{eqnarray}
	which has radius $r$ and centered at $r\cos\theta e$. Here $e:=(0,\cdots,0,-1)$ is the unit outward normal vector of $\p{\R^{n+1}_+}$ in $\ov{\R^{n+1}_+}$. If without confusion, we just write $\C_{\theta, r}$ for $\C_{\theta, r}(e)$ in the rest of the paper. 

Very recently, Wang-Weng-Xia \cite{WWX1} introduced a class of new quermassintegrals $\V_{k,\theta}(\widehat{\S})$ for capillary hypersurfaces $\S\subset \ov {\RR^{n+1}_+}$ from the variational viewpoint, where $\widehat{\S}$ is the bounded domain in $\ov{\R^{n+1}_+}$ enclosed by $\S$ and $\p {\R^{n+1}_+}$. They are, for $1\leq k\leq n$, \begin{eqnarray}\label{quermassintegrals}\V_{k+1,\theta}(\widehat{\S}):=\frac{1}{n+1}\left(\int_\S H_kdA - \frac{\cos\theta \sin^k\theta  }{n}\int_{\p\S} H_{k-1}^{\p\S}ds \right),\end{eqnarray}where  $H_k$ is the normalized $k$-th mean curvature of $\S\subset \ov{\RR^{n+1}_+}$ and $H_{k-1}^{\p\S}$ is the normalized  $(k-1)$-th mean curvature of $\p\S\subset \RR^n$ (see also  Section  \ref{sec2.1} or  \cite[Section 2.2]{WWX1}). Besides,  
	\begin{eqnarray}\label{V_01}
	\V_{0,\theta}(\widehat{\Sigma}):= |\widehat{ \S}|,~~ \V_{1,\theta}(\widehat{\Sigma}):= \frac{1}{n+1} ( |\S|-\cos\theta |\widehat{\p\S}|),
	\end{eqnarray}where $\wh{\p\S}$ is the bounded domain enclosed by $\p \S$ in  $ \p \ov{\RR^{n+1}_+}$. $\V_{1,\theta}$ is the so-called \textit{capillary area} of $\S$ (see e.g. \cite{Finn}), up to a multiple constant.

The flow is constructed for  a family of  capillary hypersurfaces  in $\ov{{\RR}^{n+1}_{+}}$ given by isometric embeddings $x(\cdot, t): M \rightarrow  \ov{\RR^{n+1}_{+}}$ from a compact $n$-dimensional manifold $M$ with   boundary $\partial M$ ($n\geq 2$) such that
	\begin{eqnarray*}
		{\rm 	int}(\Sigma_{t})=x\left({\rm int}(M), t\right)\subset\mathbb{\RR}^{n+1}_{+},\quad \partial\Sigma_{t}=x(\partial M,t)\subset \partial {\mathbb{\RR}_{+}^{n+1}}, 
	\end{eqnarray*}
	and $x(\cdot, t)$ satisfying
	\begin{equation}\label{flow with capillary}
		\left\{ \begin{array}{llll}
			(\partial_t x)^\perp  &=& \vspace{2mm}\left( 1+\cos\theta\<\nu, e\> -\<x, \nu\> F \right) \nu ,  &
			\hbox{ in }M\times[0,T^*),\\
			\<\nu,e \>& =&  -\cos\theta 
			 & \hbox{ on }\partial M \times [0,T^*),%\\			& x(\cdot,0)  = x_0(\cdot)  & \text{ on }   M\times\{0\}.
		\end{array}\right.
	\end{equation}	with $F=H_k^{\frac 1 k}$ and $T^* \in (0,\infty]$. 

 Other than flows considered in \cite{WWX1} and \cite{HWYZ}, which are inverse curvature type flows,
 flow \eqref{flow with capillary} is a curvature-type flow.
In particular, when $k=1$, i.e., $F= H\slash  n$,  flow \eqref{flow with capillary} has been studied recently by the authors together with Wang in \cite{MWW}. When $k=n$, $F=K^{\frac 1 n}$ is the $n$-th root of the Gauss curvature, \eqref{flow with capillary} is the Gauss curvature type flow.
 
 When $k=1$,  flow \eqref{flow with capillary} reduces to a mean curvature type flow, which is volume-preserving and capillary area-decreasing for functional \eqref{V_01}. However,
for $k\geq 2$,  flow \eqref{flow with capillary} generally does not preserve the enclosed volume or any quermassintegral (say \eqref{quermassintegrals} for instance). Nevertheless, a nice feature of this flow is that it decreases a  \textit{high-order capillary isoperimetric ratio}, which is defined by

\begin{eqnarray}\label{iso ratio}
    \I_{k,\theta}(\widehat{\S}):=\frac{\V_{k-1,\theta} (\widehat{\S}) }{\V_{0,\theta}(\widehat\S)^{\frac{n+2-k}{n+1}}},
\end{eqnarray}for $2\leq k\leq n+1$, see Proposition \ref{monotonicity}.

Another new ingredient in this paper is that the convexity property is preserved along the general locally constrained curvature flows \eqref{flow with capillary}, which include the case of the mean curvature type flow in \cite{MWW} when $k=1$ and Gauss curvature type flow when $k=n$.

The main result of the paper is the following.
	\begin{thm}\label{thm1.1}
	Let $\S_0$ be a strictly convex capillary hypersurface in $\ov{\RR^{n+1}_+}$ and   $\theta\in(0,\frac{\pi}{2}]$, then the solution of flow \eqref{flow with capillary} starting with  $\S_0$ exists for all time. Moreover, 
		$x(\cdot, t)$ smoothly converges to a uniquely determined spherical cap $\C_{\theta, r_{0}}$ around $e$, as $t\to\infty$.		
	\end{thm}

We focus on the case $\theta\in(0, \frac{\pi}2)$ in the rest of this paper, since $\theta=\frac{\pi}2$ can be reduced to the closed counterpart as studied in \cite{GL15, GLW} and \cite{WXiong, WXiong2}, among more general ambient space and anisotropic setting respectively,  by a simple reflection argument along $\p\RR^{n+1}_+$. The contact angle restriction $\theta\leq \frac \pi 2$ is required due to the similar technical issues as in \cite{WWX1, HWYZ} and \cite{WeX21} for obtaining the $C^2$-estimates. In fact, \eqref{neumann  bar H} is the only place that we need to use this angle restriction.

	As an application of Theorem \ref{thm1.1} and the monotonicity of \eqref{iso ratio} along our flow (cf. Proposition \ref{monotonicity}), we have the following Alexandrov-Fenchel inequality for \eqref{quermassintegrals}.
\begin{cor}\label{Alex} For $n\ge 2$,  let $\Sigma\subset    \ov{\mathbb{R}^{n+1}_+}$ be a convex capillary hypersurface with   $\theta \in (0, \frac{{\pi}}{2} ]$, there holds 
\begin{eqnarray}\label{af ineq}		
 	\left( \frac{  \V_{k,\theta}(\widehat{\S}) }{\bt} \right)^{\frac{1}{n+1-k}}\geq \left(\frac{\V_{0,\theta}(\widehat{\S})}{\bt} \right)^{\frac{1}{n+1}},  \quad  \forall \,  1 \leq k\leq n,		\end{eqnarray}where $(n+1) \bt$ is the capillary area of $\C_{\theta,1}$. 
Moreover, equality holds in \eqref{af ineq} if and only if $\Sigma$  is a spherical cap in \eqref{sph-cap}. 	
\end{cor}
We remark that \eqref{af ineq} has been previously proved in \cite{MWW} using the mean curvature type flow. We also refer to \cite{HWYZ} and \cite{MWWX}  using different methods for this, among more general forms of Alexandrov-Fenchel inequalities with the contact angle range $\theta\in(0,\frac \pi 2]$ and $\theta\in(0,\pi)$ respectively. In this paper, we provide a fully nonlinear curvature-type flow approach to proving \eqref{af ineq}. Furthermore, it is worth mentioning that $\C_{\theta, r}$ are the static solution to our flow \eqref{flow with capillary}, and conversely is also true. That is, $\C_{\theta, r}$ encompasses all solutions of
\begin{eqnarray}\label{static eq0}
1 + \cos\theta \<\nu, e\> - H_k^{\frac{1}{k}} \<x, \nu\> = 0.
\end{eqnarray}In fact, this can be easily demonstrated using either the Minkowski integral formula, as shown by  Simon \cite{Simon} for closed hypersurface (i.e. $\theta=\frac \pi 2$ in \eqref{static eq0}), or the capillary Heintze-Karcher's inequality, as discussed by Jia-Wang-Xia-Zhang \cite{JWXZ23, JWXZ} and Wang-Xia \cite{WX24}, respectively.  

 \vspace{.2cm}
 
\subsection*{The rest of the article is structured as follows} In Section \ref{sec2}, we collect some properties of elementary symmetric functions, and basic facts about capillary hypersurfaces and derive the corresponding evolution equations for some geometric quantities along flow \eqref{flow with capillary}. In Section \ref{sec3},  we reduce flow \eqref{flow with capillary} to a scalar parabolic PDE with an oblique type boundary value condition in \eqref{scalar flow with capillary-1}. Then we establish the uniform a priori estimates for the solution of \eqref{flow with capillary}. Furthermore, we prove the long-time existence and the convergence of flow \eqref{flow with capillary}, which finishes the proof of Theorem \ref{thm1.1}.  Corollary \ref{Alex} follows from Theorem \ref{thm1.1} and Proposition \ref{monotonicity}. 
\vspace{.2cm}

\section{Preliminaries}\label{sec2}
In this section,  we collect some basic properties about elementary symmetric functions and recall some basic facts of capillary hypersurface in half-space.  Then, we derive the relevant evolution equations of flow \eqref{flow with capillary}, which will be used to derive the a priori estimates later. 
	
\subsection{Elementary symmetric functions}\label{sec2.1}
	
For any $k=1,2, \cdots, n$, and $\kappa:=(\kappa_{1}, \cdots, \kappa_{n})\in \R^n$, the $k$-th elementary symmetric polynomial functions is defined as
	\begin{eqnarray*}
		\s_{k}(\kappa):=\sum\limits_{1\leq \kappa_{1}<\cdots\kappa_{k}\leq n}\kappa_{1}\cdots\kappa_{k},
	\end{eqnarray*}
	and $\sigma_{0}:=1,~~H_{k}:=\frac{1}{\binom{n}{k}}\s_k(\k).$ 	Recall the G{\aa}rding's cone is defined as
	\begin{eqnarray*}
    \Gamma_k:= \left\{\kappa\in \mathbb{R}^{n}: H_j >0,~~\forall~~1\leq  j\leq k \right\}.
	\end{eqnarray*}
Denote  $\sigma _k (\kappa \left| i \right.)$ the symmetric function with $\kappa_i = 0$ and $\sigma _k (\kappa \left| ij \right.)$ the symmetric function with $\kappa_i =\kappa_j = 0$.
	
The following properties about $\s_k$, including Lemma \ref{lem2.1},  Lemma \ref{lem2.2} and Lemma \ref{lem2.3}, are well-known, one can refer to \cite[Chapter XV, Section 4]{Lie} for a detailed proof.
\begin{lem}\label{lem2.1}
		For $\kappa=(\kappa_{1}, \cdots, \kappa_{n})$ and $1\leq k\leq n$, there holds
		\begin{enumerate}
			\item $ \sigma_k(\k)=\sigma_k(\k|i)+\k_i\sigma_{k-1}(\k|i),  \quad  \forall 1\leq i \leq n.$
			\item $ \sum\limits_{i = 1}^n {\s_{k}(\k|i)}=(n-k)\sigma_k(\k).$
			\item $\sum\limits_{i = 1}^n {\k_i \sigma_{k-1}(\k|i)}=k\sigma_k(\k)$.
			\item $\sum\limits_{i = 1}^n {\k_i^2 \sigma_{k-1}(\k|i)}=\s_1(\k)\s_k(\k)-(k+1)\s_{k+1}(\k)$.
		\end{enumerate}
	\end{lem}
\begin{lem}\label{lem2.2}
For $\kappa \in \Gamma_k$, if $k > l \geq 0$, $ r > s \geq 0$, $k \geq r$, $l \geq s$, there holds %the generalized Newton-Maclaurin inequality as
\begin{eqnarray*} \label{1.2.6}
\left(\frac{H_{k}(\k)}{H_{l}(\kappa)}\right)^{\frac{1}{k-l}}\leq \left(\frac{H_{r}(\k)}{H_{s}(\kappa) }\right)^{\frac{1}{r-s}}.
\end{eqnarray*}	
Equality holds if and only if $\k_1=\cdots=\k_n>0$. 
		Moreover, $H_k^{\frac 1 k}(\k)$ is concave in $\Gamma_k$.
	\end{lem}
	We also view $\sigma_{k}$ as a function defined on the symmetric matrice $W:=(W^i_j)$, we  denote $\sigma _k (W \left|
	i \right.)$ the symmetric function with $W$ deleting the $i$-row and
	$i$-column and $\sigma _k (W \left| ij \right.)$ the symmetric
	function with $W$ deleting the $i,j$-rows and $i,j$-columns. 
	\begin{lem}\label{lem2.3}
		Suppose $W=(W^i_j)$ is diagonal, for $1\leq k\leq n$, there holds
		\begin{align*}
			\frac{{\partial \sigma _k (W)}} {{\partial W_i^j }} = \begin{cases}
				\sigma _{k - 1} (W\left| i \right.), ~&\text{ if } i = j, \\
				0, ~&\text{ if } i \ne j.
			\end{cases}
		\end{align*}
	\end{lem}
Now we denote $F:= H_{k}^{\frac 1 k}$ and $F^{ij}:=\frac{\partial F}{\partial h_{i}^j}$.	
	The following two Lemmas about $F=H_k^{\frac 1 k}$ (see  Andrews-McCoy-Zheng \cite[Lemma 5 and Lemma 8]{Ben} resp.) are essential for us in this paper, which will be used to show that the uniform bounds of $F$ and the strict convexity preserving along our flow \eqref{flow with capillary} resp. 
	\begin{lem}[\cite{Ben}]\label{lem2.4}
	 If $\kappa\in \Gamma_{n}$, there holds
		\begin{eqnarray}\label{ineq1}
		\sum_i	F^{ii}\kappa_{i}^{2}\geq F^{2}.
		\end{eqnarray}
	\end{lem}
	
	\begin{lem}[\cite{Ben}]\label{lem2.5}
		If $\k\in \Gamma_n$, 
		then for  any $1\leq i\leq n$,  there holds
		\begin{eqnarray}\label{bdry ineq}
			(\k_1-\k_i)\left(\frac{F^{ii}}{\k_1^2}- \frac{F^{11}}{\k_i^2}\right)\leq 0.
		\end{eqnarray}
	\end{lem}
We remark that the properties \eqref{ineq1} and \eqref{bdry ineq}  hold for a large class of curvature function $F$, say for instance all the inverse concave functions defined in $\Gamma_{n}$ (cf. Andrews-McCoy-Zheng \cite{Ben}).

	\subsection{Capillary hypersurface}\
	
	Let  $\S\subset\ov{{\RR}^{n+1}_+}$ be a smooth, properly embedded capillary hypersurface, given by the embedding $X: M \to \ov{\RR^{n+1}_+}$, where $M$ is a compact, orientable smooth manifold of dimension $n$ with non-empty boundary. 
	If there is no confusion, we do not distinguish $\Sigma$ 
	and the embedding $X$.
	Let $\mu$ be the unit outward co-normal of $\p\S$ in $\S$ and  $\overline{\nu}$ be the unit normal to $\partial\Sigma$ in $\partial\mathbb{R}^{n+1}_+$ such that $\{\nu,\mu\}$ and $\{\overline{\nu},e\}$ have the same orientation in normal bundle of $\partial\Sigma\subset \ov{\mathbb{R}^{n+1}_+}$. The 
 contact angle $\theta\in(0,\pi)$ between  the hypersurface $\Sigma$ and the support $\partial\R^{n+1}_+$ is defined as
	$$\< \nu, e\>=\cos (\pi-\theta),~~~~\p \S.$$
Along $\p\S$,  
	\begin{eqnarray*}\label{co-normal bundle}
		\begin{array} {rcl}
			e &=&\sin\theta \mu-\cos\theta \nu,
			\\
			\overline{\nu} &=&\cos\theta \mu+\sin\theta \nu.
		\end{array}
	\end{eqnarray*}
	
	We use $D$ to denote the Levi-Civita connection of $\ov{\RR^{n+1}_+}$ w.r.t the Euclidean metric $\delta$, 
	and $\n$  the Levi-Civita connection on $\S$ w.r.t the induced metric $g$ from the immersion $X$. 
	The operators $\div, \Delta$, and $\n^2$ are the divergence, Laplacian, and Hessian operators on $\S$ respectively.
	The second fundamental form $h$ of $X$  is defined by
	$$D_{U}V=\n_{U}V- h(U,V)\nu,$$ for any $U,V\in T\S$.
	Denote $\k:=(\k_1, \k_2,\cdots, \k_n)\in \R^n$ be the set of principal curvatures, i.e., the eigenvalues of Weingarten curvature tensor $h:=(h^i_j)$.  
We use semi-common $f_{;i}:=\nabla_{e_{i}}f, f_{;ij}:=\nabla^{2}f(e_{i},e_{j})$, etc., to indicate the covariant derivative with respect to   $g$ and an orthonormal frame $\{e_i\}_{i=1}^n$ over $\S$, and we use the Einstein summation convention 
	that the repeated indices are implicitly summed over, no matter the indices are upper or lower. 

\subsection{Evolution equations}\label{sec2.3}\
	
	In this subsection, we derive the evolution equations for some geometric quantities along flow \eqref{flow with capillary}. 	Let $\Sigma_t$ be a family of smooth, embedding hypersurfaces with capillary boundary in $\ov{ \mathbb{R}^{n+1}_+}$, given by embeddings $x(\cdot,t):M \to \ov{\mathbb{R}^{n+1}_+}$, which evolves by flow
	\begin{eqnarray}\label{flow with normal and tangential}
		\p_t x=f\nu+\T,
	\end{eqnarray} where  $\T\in T\Sigma_t$ is some tangential vector field.  For our flow \eqref{flow with capillary}, we have  $$f=1+\cos\theta \<\nu,e\>-\<x,\nu\> F,$$ and the choice of $\T$ satisfies that the restriction of $x(\cdot, t)$ on $\p M$ is contained in $\p\R^{n+1}_+$, more precisely, $
	\T |_{\p \S_t}=f\cot\theta \mu$, see  \cite[Section 2]{WWX1} for more discussion. Along flow \eqref{flow with normal and tangential}, we have the following evolution equations for the induced metric $g_{ij}$, the
	unit outward normal $\nu$, the second fundamental form $(h_{ij})$, the Weingarten
	tensor  $(h^i_j)$, the $k$-th mean curvature $\sigma_k$ and  $F:=F(h^j_i)$ of the hypersurfaces $\Sigma_t$. (cf.   \cite[Proposition 2.11]{WeX21} for a proof). 
	
	\begin{prop}[\cite{WeX21}]\label{basic evolution eqs}
		Along  flow \eqref{flow with normal and tangential}, there hold
		\begin{enumerate} 
			\item $\p_t g_{ij}=2fh_{ij}+\n_i \T_j+\n_j\T_i$.
		%	\item $\p_tdA_t =\left(fH+\div(T)\right)dA_t.$ 
			\item $\p_t\nu =-\n f+h(e_i,\T)e_i$.	
			\item $\p_t h_{ij}=-\n^2_{ij}f +f(h^2)_{ij} +\n_\T h_{ij}+h_{j}^k\n_i\T_k+h_{i}^k\n_j \T_k,$ where $(h^2)_{ij}=h_{ik}h_{j}^k$.
			\item $\p_t h^i_j=-\n^i\n_{j}f -fh_{j}^kh^{i}_k+\n_\T h^i_j.$
\item $\p_t \sigma_k=- \frac{\p \s_k}{\p h_i^j}\n^i\n_{j}f - f (\s_1\s_k-(k+1)\s_{k+1})+\langle \n \s_k, \T\rangle $.
			\item  $\partial_t F= -F^{j}_i\n^i\n_{j}f -fF^{j}_ih_{j}^kh^{i}_k+\langle\n F,\T\rangle $,  where $F^i_j:=\frac{\partial F}{\p h^j_i}$.
			
		\end{enumerate}  
	\end{prop}

For simplicity, we introduce the linearized parabolic operator with respect to flow \eqref{flow with capillary},
\begin{eqnarray*}
	\mathcal{L}:=\partial_{t}-\<x, \nu\>F^{ij}\nabla^{2}_{ij} -\<\T+Fx-\cos\theta e, \nabla\>.
\end{eqnarray*}
%where $F^{ij}:=\frac{\partial F}{\partial h_{i}^{j}}$ and $F:=H_k^{\frac 1 k}$. 	
	\begin{prop}
		Along flow \eqref{flow with capillary},   the support function $u:=\<x, \nu\>$ satisfies
		\begin{eqnarray}\label{evo of u}
			\mathcal{L}u= 1+\cos\theta\<\nu, e\>-2uF +u^{2}F^{ij}(h^{2})_{ij},
		\end{eqnarray}
		and 
		\begin{eqnarray}\label{deri of u}
			\nabla_{\mu}u=\cot\theta h(\mu, \mu)u, \quad {\rm on}~\partial\Sigma_{t}.
		\end{eqnarray}
	\end{prop}
	\begin{proof}	For equation \eqref{deri of u}, one can refer to \cite[Proposition 3.1]{MWW}. We only need to show \eqref{evo of u} in the following. 	Direct computation yields
		\begin{eqnarray*}
			u_{;i}=h_{ik}\<x, e_{k}\>,
		\end{eqnarray*}
		and
		\begin{eqnarray*}
			u_{;ij}= h_{ij;k}\<x, e_{k}\>+h_{ij}-u(h^{2})_{ij},
		\end{eqnarray*}
		then  
		\begin{eqnarray}\label{Fij u_ij}
			F^{ij}u_{;ij}=\<x, \nabla F\>+F-uF^{ij}(h^{2})_{ij}.
		\end{eqnarray}
		Combining with Proposition \ref{basic evolution eqs} (2), we see
		\begin{eqnarray*}
			\partial_{t}u&=&\<\p_{t}x, \nu\>+\<x, \partial_{t}\nu\>=f-\<x, \nabla f\>+h(x^{T}, \T)\\
			&=&\left(1+\cos\theta\<\nu, e\>-uF\right)-\cos\theta \left<x, \nabla \<\nu, e\>\right>+u\<x, \nabla F\>\\
			&&+F\<x, \nabla u\>+h(x^{T}, \T)\\
			&=&\<x, \nu\>F^{ij}u_{; ij}-uF+u^{2}F^{ij}(h^{2})_{ij} + \left(1+\cos\theta\<\nu, e\>-uF\right)\\
			&& -\<\cos\theta e, \nabla u \>+F\<x, \nabla u\>+\<\T, \nabla u\>,
		\end{eqnarray*}where we have used \eqref{Fij u_ij} in the last line. Hence it follows 
		\begin{eqnarray*}
			\mathcal{L}u= 1+\cos\theta\<\nu, e\>-2uF+u^{2}F^{ij}(h^{2})_{ij}.
		\end{eqnarray*}
		
\end{proof}
	
Recall that in \cite{MWW} (see also \cite{WWX2}), we have  the \textit{capillary support function}  as
	\begin{eqnarray*}\label{relative support}
		\bar u:=\frac{\<x, \nu\>}{1+\cos\theta \<\nu,e\>}.
	\end{eqnarray*}

	\begin{prop}\label{evo of relative u}
		Along   flow \eqref{flow with capillary},    $\bar u$ satisfies
		\begin{eqnarray}\label{evo of bar u}
			\L \bar u=1-2\bar u F+\bar u^2 F^{ij}(h^{2})_{ij}+ 2u F^{ij}\bar{u}_{;i}(\cos\theta\<\nu, e\>)_{;j},
		\end{eqnarray}
		and  \begin{eqnarray}\label{neumann of bar u}
			\n_{\mu} \bar u=0, \quad \text{ on } \p\S_t.
		\end{eqnarray}  
	\end{prop}	
	
	\begin{proof}Since \eqref{neumann of bar u} has been shown in   \cite[Proposition 3.3]{MWW}. We only derive \eqref{evo of bar u} in the following. 
		From the Codazzi formula, 	\begin{eqnarray}\label{eq_b}
			\langle \nu,e\rangle_{;ij}= h_{ij;k}\langle e_k,e\rangle -(h^2)_{ij}\langle \nu,e\rangle ,
		\end{eqnarray}
		it yields
		\begin{eqnarray*}
			F^{ij}\<\nu,e\>_{;ij}= \<\n F,e\>-\<\nu,e\>F^{ij}(h^{2})_{ij}.
		\end{eqnarray*}
		Combining  it with
		\begin{eqnarray*}
			\<\n f,e\>=\cos\theta h(e^T,e^T)-\<x,\nu\>\< \n F,e\>-Fh(x^T,e^T),
		\end{eqnarray*}
		we obtain
		\begin{eqnarray*}
			\p_t \<\nu,e\>&=& -\<\n f,e\>+h(e_i,\T)\<e_i,e\>
			\\&=&-\cos\theta h(e^T,e^T)+u\<\n F,e\>+Fh(x^T,e^T)+h(\T,e^T)
			\\&=& u\left( F^{ij}\<\nu,e\>_{;ij}+\<\nu,e\>F^{ij}(h^{2})_{ij}\right)-\cos\theta h(e^T,e^T) +Fh(x^T,e^T)+h(\T,e^T),
		\end{eqnarray*}
		that is,
		\begin{eqnarray}\label{ev of nu e}
			\L \<\nu,e\>= u \<\nu,e\>F^{ij}(h^{2})_{ij}.
		\end{eqnarray} 
	From \eqref{evo of u} and \eqref{ev of nu e}, it follows
		\begin{eqnarray*}
			\L \bar u&=& \frac{1}{1+\cos\theta \<\nu,e\>} \L u -\frac{\<x,\nu\>}{(1+\cos\theta\<\nu,e\>)^2} \L (1+\cos\theta \<\nu,e\>) \\
			&&-\<x,\nu\>F^{ij}\left(\frac{2u\cos^{2}\theta\<\nu_{;i}, e\>\<\nu_{;j}, e\>}{(1+\cos\theta\<\nu,e\>)^3 }-2\frac{ \<x,\nu\>_{;i} \cos\theta \<\nu_{;i},e\>}{(1+\cos\theta\<\nu,e\>)^2}\right)
			\\&=&1-\frac{2uF}{1+\cos\theta \<\nu,e\>}+\frac{u^2 F^{ij}(h^{2})_{ij}}{(1+\cos\theta\<\nu,e\>)^2}+ \frac{2u F^{ij}\bar{u}_{;i}(1+\cos\theta\<\nu, e\>)_{;j}}{1+\cos\theta\<\nu, e\>}.%\\&\geq & \left( \frac{u|H|}{\sqrt{n}(1+\cos\theta \<\nu,e\>)} -\sqrt{n}\right)^2,
		\end{eqnarray*}
	\end{proof}

	Now we compute the evolution equation for the curvature function $F=H_k^{\frac 1 k}$.
	\begin{prop}\label{evol of H}
		Along flow \eqref{flow with capillary}, we have
		\begin{eqnarray}\label{evo of H}
			\mathcal{L}F=2F^{ij}F_{;i}u_{;j}+F^{2}-F^{ij}(h^{2})_{ij},
		\end{eqnarray}
		and
		\begin{eqnarray}\label{deri of H}
			\nabla_{\mu}F=0,\quad{\rm on }~\partial\Sigma_{t}.
		\end{eqnarray}
	\end{prop}
	\begin{proof}
		From \eqref{eq_b}, it is easy to see that
		\begin{eqnarray*}
			f_{;ij}&=&\cos\theta h_{ij;k}\<e_{k}, e\>-\cos\theta (h^{2})_{ij}\<\nu, e\>-F_{;ij}\<x, \nu\>-F_{;i}h_{jk}\<x, e_{k}\>\\
			&&-F_{;j}h_{ik}\<x, e_{k}\>-Fh_{ij}-Fh_{ij;k}\<x, e_{k}\>+F(h^{2})_{ij}\<x, \nu\>.
		\end{eqnarray*}
		Combining it with Proposition \ref{basic evolution eqs}, we have
		\begin{eqnarray*}
			\partial_{t}F   &=&-F^{ij}f_{;ij}-fF^{ij}(h^{2})_{ij}+\<\n F, \T\>\\
			&=&-\<\cos\theta e, \n F\>+\cos\theta \<\nu, e\>F^{ij}(h^{2})_{ij}+uF^{ij}F_{;ij}+2F^{ij}F_{;j}h_{ik}\<x, e_{k}\>\\
			&&+FF^{ij}h_{ij}+\<Fx, \n F\>-uFF^{ij}(h^{2})_{ij}-\left(1+\cos\theta\<\nu, e\>\right)F^{ij}(h^{2})_{ij}\\
			&&+uFF^{ij}(h^{2})_{ij}+\<\n F, \T\>,
		\end{eqnarray*}
		note that  $F^{ij}h_{ij}=F$, then  
		\begin{eqnarray*}
			\mathcal{L}F=F^{2}-F^{ij}(h^{2})_{ij}+2F^{ij}F_{;j}u_{;i}.
		\end{eqnarray*}
It has been shown in \cite[Proposition 4.3]{WWX2} that  on $\partial\Sigma_{t}$, 
		\begin{eqnarray*}
			\n_\mu f=\cot\theta h(\mu,\mu) f.
		\end{eqnarray*}
		Hence, combining with \eqref{deri of u},  it follows
		\begin{eqnarray*}
			\n_\mu F&=&\n_\mu \left(\frac{1+\cos\theta \langle \nu,e\rangle -f}{\langle x,\nu\rangle }\right) =0. 
		\end{eqnarray*}
	\end{proof} 
	
	For a strictly convex hypersurface $\S$, denote  $(b^{ij})$ the inverse of $(h_{ij})$, and   $$\bar H:=\sum_{i=1}^n \frac{1}{\k_i}=\sum_{i,j=1}^ng_{ij}b^{ji},$$
	which is the harmonic curvature or the sum of the principal curvature radii. The harmonic curvature satisfies the following evolution equation, together with a nice boundary inequality, provided that $\theta \in (0, \frac \pi 2]$. 
	
	\begin{prop}\label{evo of bar H}
		Along flow \eqref{flow with capillary}, $\bar H$ satisfies
		\begin{eqnarray*}
		\mathcal{L}\bar{H}&=&-2uF^{kl}b^{ip}h_{pq;l}b^{qs}h_{sr;k}b_{i}^{r}-ub_{s}^{i}b_{i}^{r}F^{kl, pq}h_{kl;s}h_{pq;r}-2b_{s}^{i}\<x, e_{i}\>g^{sr}F_{;r}\\
			&&+nu F+n-u\bar{H}F^{kl}(h^{2})_{kl}-F\bar{H}.
		\end{eqnarray*}\label{evolu of bar H}
		If $\theta\in (0,\frac{\pi}{2}]$, there holds
		\begin{eqnarray}\label{neumann  bar H}
			\nabla_{\mu}\bar{H}\leq 0, \quad {\rm on}~\partial M.
		\end{eqnarray}
	\end{prop}
	\begin{proof}
		Direct computation yields
		\begin{eqnarray*}
			\nabla_{k}\bar{H}=-b^{is}\n_k h_{sr} b^r_i,
		\end{eqnarray*}
		and
		\begin{eqnarray}\label{hessian of bar H}
			\n^l\n_k\bar{H}&=&-b^{is}\n^l\n_k h_{sr}b^r_i+2b^{ip} \n^l h_{pq}b^{qs} \n_k h_{sr}b^r_i.
		\end{eqnarray}
		From Codazzi equations, Ricci identities, and Gauss equation, we have\begin{eqnarray*}
			h_{kl;sr}&=&h_{ks;lr}=h_{ks;rl}+R^p_{slr}h_{pk}+R^p_{klr}h_{ps}
			\\&=&h_{sk;rl}+R^p_{slr}h_{pk}+R^p_{klr}h_{ps}
			\\&=&h_{sr;kl}+(h_{pl}h_{sr}-h_{pr}h_{ls})h_{pk}+(h_{pl}h_{kr}-h_{pr}h_{kl})h_{ps},
		\end{eqnarray*}and thus
		\begin{eqnarray}\label{ricci}
			b^{is}b_i^r F^{kl} h_{sr;kl }= b^{is}b_i^r F^{kl}h_{kl;sr} -\bar H F^{kl}(h^2)_{kl}+nF.
		\end{eqnarray}
		On the other hand,
		\begin{eqnarray*}
			\p_t \bar H&=&\p_t b^i_i=-b^{i }_s\p_t h_{r}^s b^r_i
			\\&=&-b^{i}_s\left(-\n^s\n_{r}f -fh_{r}^ph^{s}_p+\n_\T h^s_r\right) b^l_i
			\\&=&b^i_s b^r_i \n^s\n_r f+nf+\n_\T\bar H.
		\end{eqnarray*}
		From the Codazzi formula, we have
		\begin{eqnarray*}
			\langle x,\nu\rangle_{;kl}&=  h_{kl}+h_{kl;s}\langle x,e_s\rangle -(h^2)_{kl}\langle x,\nu\rangle ,
		\end{eqnarray*}
		and
		\begin{eqnarray*}
			\langle \nu,e\rangle_{;kl}= h_{kl;s}\langle e_s,e\rangle -(h^2)_{kl}\langle \nu,e\rangle ,
		\end{eqnarray*}
		it follows
		\begin{eqnarray*}
			f_{;sr}&=& \n^s\n_r ( {1+ \cos\theta \<\nu,e\>}-{F} \<x,\nu\>)
			\\&=& \cos\theta   \Big(h_{sr;j}\langle e_j,e\rangle -(h^2)_{sr}\langle \nu,e\rangle\Big)-   \left(h_{sr}+h_{sr;j}\langle x,e_j\rangle -(h^2)_{sr}\langle x,\nu\rangle\right)F\\&& -\<x,\nu\> \n^s\n_r F-\n^s F\n_r \<x,\nu\>-\n_r F\n^s \<x,\nu\>,
		\end{eqnarray*}
		combining with
		\begin{eqnarray*}
			F_{;sr}=F^{kl,pq}h_{kl;s}h_{pq;r}+F^{kl}h_{kl;sr},
		\end{eqnarray*}
		it gives
		\begin{eqnarray*}
			b_{s}^{i}b_{i}^{r}f_{; sr}&=&-\<x, \nu\>b_{s}^{i}b_{i}^{r}\left(F^{kl}h_{kl; sr}+F^{kl, pq}h_{kl; s}h_{pq; r}\right)-\<\cos\theta e, \nabla\bar{H}\>-n\cos\theta \<\nu, e\>\\
			&&-F\bar{H}+\<Fx, \nabla \bar{H}\>+n\<x, \nu\>F-2b_{s}^{i}\<x, e_{i}\>\nabla^{s}F.
		\end{eqnarray*}
		Plugging   \eqref{hessian of bar H} and \eqref{ricci} into $\p_t \bar H$, we have 
		\begin{eqnarray*}
			\partial_{t}\bar{H}&=&-\<x, \nu\>b_{s}^{i}b_{r}^{i}F^{kl}h_{kl;sr}-\<x, \nu\>b_{s}^{i}b_{i}^{r}F^{kl, pq}h_{kl;s}h_{pq;r}-2b_{s}^{i}\<x, e_{i}\>\n^{s}F\\
			&&+\<\T+Fx-\cos\theta e, \n \bar{H}\>+n-F\bar{H}\\
			&=&u F^{kl}\n_k\n_{l}H-2uF^{kl}b^{ip}\n^{l}h_{pq}b^{qs}\n_{k}h_{sr}b_{i}^{r}-\<x, \nu\>\bar{H}F^{kl}(h^{2})_{kl}\\
			&&+nu F
			-\<x, \nu\>b_{s}^{i}b_{i}^{r}F^{kl, pq}h_{kl;s}h_{pq;r}-2b_{s}^{i}\<x, e_{i}\>\n^{s}F\\
			&&+\<\T+Fx-\cos\theta e, \n \bar{H}\>+n-F\bar{H}.
		\end{eqnarray*}
		It follows
		\begin{eqnarray*}
			\mathcal{L}\bar{H}&=&-2uF^{kl}b^{ip}h_{pq;l}b^{qs}h_{sr;k}b_{i}^{r}-ub_{s}^{i}b_{i}^{r}F^{kl, pq}h_{kl;s}h_{pq;r}-2b_{s}^{i}\<x, e_{i}\>g^{sr}F_{;r}\\
			&&+nu F+n-u\bar{H}F^{kl}(h^{2})_{kl}-F\bar{H}.
		\end{eqnarray*}
		
		Along $\p\S_t$, choosing an orthonormal frame $\{e_\alpha\}_{\alpha=2}^{n}$ of $T{\p\S_t}$ such that $\{e_1:=\mu,(e_\alpha)_{\alpha=2}^n\}$ forms an orthonormal frame for $T\S_t$.
		\begin{eqnarray*}\label{neumanof F}
			\n_\mu F=0, \quad \text{ i.e., } ~F^{11}h_{11;1}+F^{\alpha\alpha}h_{\alpha\alpha;1}=0,
		\end{eqnarray*}
		then
		\begin{eqnarray*}
			h_{11;1}=-\frac{ F^{\alpha\alpha}h_{\alpha\alpha;1}}{F^{11}}.
		\end{eqnarray*}
		Recall that $\k_i=h_{ii}$ and $\k\in \Gamma_n$, using Lemma \ref{lem2.5}, it yields $$(h_{11}-h_{\alpha\alpha})\left[(b^{11})^{2}F^{\alpha\alpha}-(b^{\alpha\alpha})^{2} F^{11}\right]=(\k_1-\k_\alpha)\left(\frac{F^{\alpha\alpha}}{\k_1^2}- \frac{F^{11}}{\k_\alpha^2}\right)\leq 0.$$ 
		From \cite[Proposition 2.3 (3)]{WWX2}, we know
		\begin{eqnarray*}
			h_{\alpha\alpha;\mu}=\cot\theta h_{\alpha\alpha}(h_{11}-h_{\alpha\alpha}),
		\end{eqnarray*}
		combining with $\theta\in(0,\frac{\pi}{2}]$,  
		\begin{eqnarray}
			\nabla_{\mu}\bar{H} \notag &=&-(b^{11})^{2}h_{11;\mu}-(b^{\alpha\alpha})^2 h_{\alpha\alpha;\mu}\\ \notag
			&=&\frac{1}{F^{11}}\left[(b^{11})^{2}F^{\alpha\alpha}-(b^{\alpha\alpha})^{2} F^{11}\right]h_{\alpha\alpha;\mu}\\ \notag
			&=&\cot\theta h_{\alpha\alpha}(h_{11}-h_{\alpha\alpha})\left[(b^{11})^{2}F^{\alpha\alpha}-(b^{\alpha\alpha})^{2} F^{11}\right]
			\\&\leq &0. \notag %\label{angle restrict}
		\end{eqnarray}
	\end{proof}

\section{Long-time existence and convergence}\label{sec3}

As the initial capillary hypersurfaces $\S_0$ of strictly convex in $\ov{ \R^{n+1}_+}$, the short-time existence follows by adapting the argument as in  \cite{HP1999}. Let $T^*>0$ be the maximal existence time of convex solution to flow \eqref{flow with capillary}. We firstly reduce flow \eqref{flow with capillary} to a scalar parabolic flow \eqref{scalar flow with capillary-1}, then we derive uniform a priori estimates, and show that the convexity is preserved along flow \eqref{flow with capillary}. Finally, we finish the proof of Theorem \ref{thm1.1}.

\subsection{A scalar flow}\

Assume that  a capillary hypersurface 
$\S\subset \ov{\RR^{n+1}_+}$ is  star-shaped with respect to the origin. 
One can reparametrize it as a graph over $\overline{\mathbb{S}}^n_+$. Namely, 
there exists a positive function $\rho$ defined on $\overline{\SS}^n_+$ such that
\begin{eqnarray*}
	\S=\left\{ \rho ( X)X ~ |    X\in \overline{\SS}^n_+\right\},
\end{eqnarray*}where $X:=(X_1,\ldots, X_n)$ is a local coordinate of $\overline{\SS}^n_+$.
We  denote  $\bar\n$ as the   Levi-Civita connection on $\SS^n_+$ with respect to the standard round metric $\sigma:=g_{_{\SS^n_+}}$, $\p_i:=\p_{X_i}$,  $\sigma_{ij}:=\sigma(\p_{i},\p_{j})$,  $\rho_i:=\bar\n_{i} \rho$, and $\rho_{ij}:=\bar\n_{i}\bar\n_j \rho$.

In order to express the capillary boundary condition in terms of the radial function $\rho$, we use the polar coordinate in 
the half-space. For $x:=(x',x_{n+1})\in \RR^n\times [0,+\infty)$  and $X:=(\beta,\xi)\in [0,\frac{\pi}{2}]\times \SS^{n-1}$, we have that 
\begin{eqnarray*}
x_{n+1}=r\cos\beta,\quad ~~ |x'|=r\sin\beta.
\end{eqnarray*} 
Then
\begin{eqnarray*}
	e_{n+1}=\cos\beta \p_r-\frac{\sin\beta}{r}\p_\beta.
\end{eqnarray*} 
In these coordinates, the standard Euclidean metric is given by
$$ |dx|^2=dr^2+r^2\left(d\beta^2+\sin^2\beta g_{_{\SS^{n-1}}}\right),$$
and
\begin{eqnarray*}
	\langle \nu,e_{n+1}\rangle =\frac{1}{v}\left(\cos\beta+\sin\beta \bar\n_{\p_\beta}    \varphi \right).
\end{eqnarray*}

 Define $\varphi(X,t):=\log \rho(X,t)$ and $v:=\sqrt{1+|\bar\nabla\varphi|^{2}}$, for $X\in \ol{\SS}^n_+$,  then the flow \eqref{flow with capillary} can be reduced to a scalar parabolic PDE with oblique boundary value condition (see e.g., \cite{WWX1, WX20}). That is,
\begin{eqnarray}\label{scalar flow with capillary-1}
\left\{	\begin{array}{rcll}
\p_t \varphi &=&\frac{v}{e^{\varphi}}\widehat{f},\quad &\text{ in } \SS^n_+\times [0,T^*),  \\
\bar\n_{\p_\beta} \varphi&=&\cos\theta \sqrt{1+|\bar\n\varphi|^2},\quad &\text{ on }
\p\SS^n_+\times [0,T^*),\\
\varphi(\cdot,0)&=&\varphi_0(\cdot),\quad &\text{ on } \SS^n_+,
\end{array}\right.
\end{eqnarray}
where $\varphi_0 $ is the parameterization radial function of $x_0(M)$ over $\overline{\SS}^n_+$, $\p_\beta$ the unit outward normal of $\p \SS^n_+$ in $\overline \SS^n_+$, and
\begin{eqnarray*}
	\widehat{f}:=\left(1-\frac{\cos\theta}{v}(\cos\beta+\sin\beta\bar\nabla_{\partial_{\beta}}\varphi)\right)-\frac{e^{\varphi}}{v} H_{k}^{\frac 1 k} .
\end{eqnarray*}

\subsection{$C^{0}$ and $C^{1}$ estimates}

The strict convexity of $\Sigma_0$ implies that there exists some $0<r_1<r_2<+\infty$, such that 
	$$\S_0\subset \widehat{\C_{\theta,r_2}}\setminus \widehat{\C_{\theta,r_1}},$$
	where $\widehat{\C_{\theta,r}}$ denotes the  bounded domain in $\ov{\R^{n+1}_+}$ enclosed by $\C_{\theta,r}$ and $\p\R^{n+1}_+$.
	Following the argument in \cite[Proposition 4.2]{WWX1}, which is based on the avoidable principle, we have 

\begin{prop}\label{upper bound of u}
		For all $t\in [0, T^*)$, along flow \eqref{flow with capillary}, there holds
		\begin{eqnarray*}\label{c0 bound}
			\Sigma_{t}\subset \widehat{\C_{\theta,r_2}}\setminus\ \widehat{\C_{\theta,r_1}}.
		\end{eqnarray*}
		where $\C_{\theta,r}$ is defined by \eqref{sph-cap} and $r_{1}, r_{2}$ only depend on $\Sigma_{0}$.
	\end{prop}

	In what follows, we indicate $C$ to be the positive constant (which may vary line from line) depending only on the initial datum $\S_0$.  	Next, we show the uniform lower bound of $u$.
	\begin{prop}	\label{lower bound of u}
		Let $\S_0$ be a strictly convex hypersurface with capillary boundary in $\ov{ \RR^{n+1}_+}$ and $\theta \in(0,\pi)$, there holds
	\begin{eqnarray*}\label{c1 est}
			\langle x,\nu\rangle(p,t) \geq C,~ ~ ~ \forall ~(p,t)\in M\times [0,T^*),
		\end{eqnarray*} where the positive constant $C$ depends only on $\S_0$.
	\end{prop}
	
	\begin{proof}
		Combining Proposition \ref{evo of relative u} and Lemma \ref{lem2.4}, we see
		\begin{eqnarray*}
			\L \bar u&=&  
			1-2\bar{u}F+\bar{u}^{2}F^{ij}(h^{2})_{ij} \quad \text{ mod} \quad \n \bar u	\\&\geq & \left( 1-\bar{u}F\right)^2\geq 0,
		\end{eqnarray*}	 together with   $\n_\mu \bar u=0$ on $\p M$, it implies	\begin{eqnarray*}
			\bar u\geq \min_{M} \bar u(\cdot,0).
		\end{eqnarray*}
		Since  $1+\cos\theta \<\nu,e\>\geq 1-\cos\theta >0$, we conclude 
		\begin{eqnarray*}
			\<x,\nu\>=\bar u\cdot (1+\cos\theta \<\nu,e\>)\geq C,
		\end{eqnarray*} for some positive constant $C>0$ independent of $t$.
	\end{proof}

\subsection{Curvature estimates}\
	
	In this section, firstly, we show that the curvature  $F:=H_k^{\frac 1 k}$ has uniform lower and upper bounds. 
 The proof argument and the test function here are similar to \cite{WWX2} and \cite{MWW}, for completeness, we contain a proof here.
	Secondly, we show that strict convexity is preserved along flow \eqref{flow with capillary}.
	\begin{prop}\label{upper bound of F}
		If $\Sigma_{t}$ solves flow \eqref{flow with capillary}, then  
			\begin{eqnarray*}
			F(p, t)\leq \max_{M}F(\cdot, 0),\quad\quad \forall~ (p, t)\in M\times [0,T^*).
		\end{eqnarray*}
	\end{prop}
	\begin{proof}
		From equation \eqref{evo of H} and $F^{ij}(h^{2})_{ij}\geq F^2$, 
		\begin{eqnarray*}
			\L F\leq 0,\quad \text{ mod } \n F,
		\end{eqnarray*}and $\n_{\mu }F =0$ on $\p M$, hence the conclusion follows directly from  the maximum principle.
	\end{proof}
	\begin{prop}\label{lower bound of F}
		If  $\Sigma_{0}$ is a strictly convex capillary hypersurface  and $\S_t$ solves flow \eqref{flow with capillary}, then %  in $\bar \RR^{n+1}_+$, then
		\begin{eqnarray*}
			F(p, t)\geq C,\quad\quad~ \forall  ~(p, t)\in M\times [0,T^*),
		\end{eqnarray*}
		where the positive constant  $C$ depends on the initial datum. % and $\cos\theta$.
	\end{prop}
	\begin{proof}
		We use the  test function $$P:=F\bar u.$$ In view of \eqref{neumann of bar u} and \eqref{deri of H},   we have \begin{eqnarray}\label{neuman of P}
			\nabla_{\mu}P= 0,\quad \text{ on } \p M. 
		\end{eqnarray}
	From \eqref{evo of bar u} and \eqref{evo of H},
		\begin{eqnarray} \label{evolution P} 
			\L P&=& \bar u\mathcal{L}F+F\mathcal{L} \bar u-2uF^{ij}\bar{u}_{;i} F_{;j} \\ \notag
			&=& -\bar{u}F^{2}+\bar{u}^{2}FF^{ij}(h^{2})_{ij}-\bar{u}F^{ij}(h^{2})_{ij}+F+2\cos\theta\bar{u}F^{ij}P_{;i}\<\nu_{;j}, e\>.\notag
		\end{eqnarray} 
		Taking \eqref{neuman of P} and the Hopf boundary Lemma into account,	if $P$ attains the minimum value at $t=0$, then we are done, since the conclusion follows directly by using the uniform bound of $\bar u$. Therefore, we assume that
		$P$ attains the minimum value at some interior point, say $p_0\in \text{int}(M)$. At $p_0$,  there hold
		$$\n P=0, \quad \L P\leq 0.$$ Substituting  it into \eqref{evolution P}, it yields
		\begin{eqnarray}\label{2nd polynoimial}
			\left( \frac 1 F-\bar u\right) \bar uF^{ij}(h^{2})_{ij}+\bar uF\geq 1.
		\end{eqnarray}	 
		
		If $\bar uF\geq 1$ at $p_0$, then we are done. Now assume that $\bar uF<1$ at $p_0$, by Lemma \ref{lem2.4}, we obtain 
		\begin{eqnarray*}
			\bar	uF(p_0, t)\geq c,
		\end{eqnarray*}for some positive constant $c$, which only depends  on $n$ and $k$. This also yields the desired estimate.
	\end{proof}
	
Next, we show that the fully nonlinear curvature flow \eqref{flow with capillary} preserves the convexity if $\theta\in (0,\frac{\pi}{2}]$.  We remark that this is the only place where we have used the restriction of contact angle range.

	\begin{prop}\label{preserve convexity}
		Let $\Sigma_{0}$ be a strictly convex hypersurface, if $\Sigma_{t}$ solves flow \eqref{flow with capillary} %with the initial value $\Sigma_{0}$ 
		with  $\theta\in (0,\frac{\pi}{2}]$, then
		\begin{eqnarray}\label{convex pre}
			\min_{1\leq i\leq n}\kappa_{i}(p, t)\geq C,~~~ \forall (p, t)\in M\times[0, T^*),
		\end{eqnarray}where the positive constant  $C$ depends on the initial datum. 
		
	\end{prop}
	\begin{proof}
Note that \eqref{convex pre} is equivalent to the uniform upper bound for $\bar{H}$. 
		If the maximum value of $\bar H$ is reached at $t=0$, then we are done. Otherwise, from $\n_{\mu}\bar H\leq 0$ on $\p M$ in equation \eqref{neumann  bar H} and the Hopf boundary Lemma, we have that $\bar{H}$   attains its maximum value at some interior point, say $p_{0}\in {\rm int}(\Sigma_t)$. 
		
		At $p_0$, by choosing an orthonormal frame $\{e_{i}\}_{i=1}^{n}$ on $T\S_t$ such that $(h_{ij})$ is diagonal, which follows that $(b^{ij})$ is also  diagonal. In view of equation \eqref{evolu of bar H}, 
		\begin{eqnarray}\notag		0&\leq & \mathcal{L}\bar{H}\\ 
				&=&\notag -2u(b^{ii})^{2}F^{kk}b^{ss}h_{sk;i}^{2}-u(b^{ii})^{2}F^{kl,pq}h_{kl;i}h_{pq; i}-2b^{ii}F_{;i}\<X, e_{i}\>\\
				&&+nuF+n-u\bar{H}F^{kl}(h^{2})_{kl}-F\bar{H}.\label{key-ine}
		\end{eqnarray}
			
	Note that for each fixed $i$,  there holds (see e.g. \cite[formula~(3.49)]{Urb})
		\begin{eqnarray}\label{key-ine-2}
			F^{kl,pq}h_{kl;i}h_{pq;i}+2F^{kk}b^{ss}h_{sk;i}^{2}\geq \frac{2}{F}F_{;i}^{2}.
		\end{eqnarray}
		Substituting \eqref{key-ine-2} into \eqref{key-ine}, at $p_{0}$, we have
		\begin{eqnarray*}
			0&\leq &\mathcal{L}\bar{H}\\
			&\leq &-\frac{2u}{F}(b^{ii})^{2}F_{;i}^{2}-2b^{ii}F_{;i}\<x,e_{i}\>+nuF+n-u\bar{H}F^{kl}(h^{2})_{kl}-F\bar{H}\\&\leq &\frac{F|x^{T}|^{2}}{2u}+nuF+n-u\bar{H}F^{kl}(h^{2})_{kl}-F\bar{H}.
		\end{eqnarray*}
		Together with Proposition \ref{lower bound of F},  Proposition \ref{upper bound of F}, and Proposition \ref{lower bound of u}, it implies  $$\bar{H}\leq C,$$ which follows the desired strictly convexity estimate \eqref{convex pre}. 
		
\end{proof} 

 \begin{rem}
     One can show that the convexity estimate \eqref{convex pre} for a large class of curvature function $F$ in \eqref{flow with capillary} using our approach. In fact, our proof also works for curvature function $F$ which is concave and inverse-concave.
 \end{rem}	

	\subsection{Long-time existence and convergence}\
	
With the uniform a priori estimates above, we are ready to prove the long-time existence and the global convergence of flow \eqref{flow with capillary}, which follows the proof of Theorem \ref{thm1.1}.

\begin{proof}[\textbf{Proof of Theorem \ref{thm1.1}}]
 From Proposition \ref{upper bound of u}, we have a uniform bound for solution $\varphi$ of \eqref{scalar flow with capillary-1}, and from  Proposition \ref{lower bound of u}, we have a uniform bound for $v$, and hence a uniform bound for  $\bar{\nabla} \varphi$.	Combining Proposition \ref{upper bound of F} and Proposition \ref{preserve convexity}, we know that the principal curvature $\k_j$ is uniformly bounded from above for all $1\leq j\leq n$, which together with uniform gradient estimates, implies that $\varphi$ is uniformly bounded in $C^2(\ov{\SS}^n_+\times[0, T^*))$ and the scalar equation in \eqref{scalar flow with capillary-1} is uniformly parabolic. Since $|\cos\theta|<1$, from the standard nonlinear parabolic theory with a strictly oblique boundary condition theory (cf.  \cite[Theorem 6.1, Theorem 6.4 and Theorem 6.5]{Dong}, also  \cite[Theorem 5]{Ura}  and \cite[Theorem 14.23]{Lie})), we conclude the uniform $C^\infty$ estimates and the long-time existence. And the convergence can be shown similarly by using the argument as in \cite[Section 8.2]{WXiong} or \cite[Proposition 4.12]{WWX1}, by adopting the monotonicity property of $\V_{0,\theta}(\widehat{\S_t})$ in  \eqref{volume-t}	along flow \eqref{flow with capillary} and the uniform a priori estimates.  We omit it here.
\end{proof}

\begin{prop}\label{monotonicity}
    Along flow \eqref{flow with capillary}, $\I_{k,\theta}(\widehat{\S_t})$ is monotone non-increasing with respect to time $t>0$.
\end{prop}

\begin{proof}
    From the Minkowski formula in \cite[Proposition 2.5]{WWX1}, for capillary hypersurfaces $\S_t$  in the half-space, there holds
\begin{eqnarray}\label{minkowski formula}
		\int_{\Sigma_{t}}H_{k-1}\left(1+\cos\theta\<\nu, e\>\right)dA= \int_{\Sigma_{t}}H_{k}\<x, \nu\>dA,
	\end{eqnarray}
together with  \cite[Theorem 2.7]{WWX1} and Newton-MacLaurin inequality $H_1\geq H_k^{\frac 1 k}$ for $2\leq k\leq n$, we derive 
    \begin{eqnarray}\label{volume-t}
        \frac{d}{dt} \V_{0,\theta}(\widehat{\S_t})&=& \int_{\S_t} \left[ (1+\cos\theta \<\nu,e\>)-H_k^{\frac 1 k} \<x,\nu\> \right ]dA\\&=&\int_{\S_t}\left( H_1 -H_k^{\frac 1 k} \right)\<x,\nu\>dA\geq 0.\notag
    \end{eqnarray}
Using  \cite[Theorem 2.7]{WWX1} again and Newton-MacLaurin inequality  $H_{k-1} \geq H_k^{\frac{k-1}{k}}$  for $k\geq 2$, 
    \begin{eqnarray*}
        \frac{d}{dt} \V_{k-1,\theta}(\widehat{\S_t})&= & \frac{n+2-k}{n+1} \int_{\S_t} H_{k-1} \left (1+\cos\theta \<\nu,e\> -H_k^{\frac 1 k} \<x,\nu\> \right) dA\\& \leq  &\int_{\S_t} \left[H_{k-1} (1+\cos\theta \<\nu,e\>) -H_k \<x,\nu\> \right] dA=0,
    \end{eqnarray*}where the last equality follows from \eqref{minkowski formula} again.

Hence $\I_{k,\theta}(\wh{\S_t})$ is monotone decreasing for the time $t$ from the above two inequalities.  
\end{proof}

As a direct consequence of Theorem \ref{thm1.1} and Proposition \ref{monotonicity}, one can obtain Corollary \ref{Alex}.
\begin{proof}[\textbf{Proof of Corollary \ref{Alex}}]\
	
Using Theorem \ref{thm1.1} and Proposition \ref{monotonicity}, we know that \eqref{af ineq} holds for the strictly convex case. When $\S$ is only convex, we use the approximation argument as in \cite{GL09} to prove it.  While the equality characterization in Corollary \ref{Alex}   can be established along the same way as \cite{SWX, WeX21, WWX1}. Hence we complete the proof.  	

\end{proof} 	

 \ 
 
\noindent\textit{Acknowledgment}: 
Part of this work was carried out while L.W. was visiting the Institut für Mathematik at Goethe Universität Frankfurt in October 2023. He would like to express his sincere gratitude to Prof. Julian Scheuer and the institute for their warm hospitality. Both authors would like to thank Professor Guofang Wang for his valuable discussions on this subject and constant support.

\printbibliography 

\end{document}